\documentclass[12pt]{article}
\usepackage{fullpage}
\usepackage{latexsym}
\usepackage{amsmath,amsfonts,amssymb,amsthm,enumerate,dsfont,color}

\newcommand{\ignore}[1]{}

\newtheoremstyle{postnum}
  {\topsep}
  {\topsep}
  {\slshape}
  {0pt}
  {\bfseries}
  {:}
  { }
  {\thmname{#1}\thmnote{ (#3)}}

\theoremstyle{postnum}

\newtheoremstyle{prenum}
  {\topsep}
  {\topsep}
  {\slshape}
  {0pt}
  {\bfseries}
  {.}
  { }
  {\thmnumber{#2}\thmname{ #1}\thmnote{ (#3)}}

\theoremstyle{prenum}

\newtheorem{theorem}{Theorem}[section]
\newtheorem{lemma}[theorem]{Lemma}

\newtheorem{proposition}[theorem]{Proposition}
\newtheorem{corollary}[theorem]{Corollary}

\newcommand{\etal}{{\em et al.~}}

\newcommand{\ZZ}{\mathbb{Z}}
\newcommand{\RR}{\mathbb{R}}
\newcommand{\QQ}{\mathbb{Q}}
\newcommand{\CC}{\mathbb{C}}

\newcommand{\ii}{\mathtt{i}}

\DeclareMathOperator{\spn}{span}

\newcommand{\HH}{\mathcal{H}}

\DeclareMathOperator{\tr}{tr}

\newcommand{\R}{\mathds{R}}

\newcommand{\ee}{\mathbf{e}}

\newif\ifnotesw\noteswtrue
   {\ifnotesw\marginpar[\hfill\(\top\)]{\(\top\)}\fi}%
   {\ifnotesw\marginpar[\hfill\(\bot\)]{\(\bot\)}\fi}

\newcommand{\mnote}[1]%
    {\ifnotesw\marginpar%
        [{\scriptsize\begin{minipage}[t]{\marginparwidth}
        \raggedleft#1%
                        \end{minipage}}]%
        {\scriptsize\begin{minipage}[t]{\marginparwidth}
        \raggedright#1%
                        \end{minipage}}%
    \fi}

\title{Fractional Revival and Association Schemes}
\author{
Ada Chan\\
\footnotesize{Department of Mathematics and Statistics, York University, Toronto, ON, Canada}
\\
\tt \footnotesize{ssachan@yorku.ca}
\and
Gabriel Coutinho\\
\footnotesize{Department of Computer Science, Universidade Federal de Minas Gerais, Belo Horizonte, MG, Brazil}
\\
\tt \footnotesize{gabriel@dcc.ufmg.br}
\and
Christino Tamon\\
\footnotesize{Department of Computer Science, Clarkson University, Potsdam, NY, USA}
\\
\tt \footnotesize{tino@clarkson.edu}
\and
Luc Vinet \ \& \ Hanmeng Zhan \\
\footnotesize{Centre de Recherches Math\'{e}matiques, Universit\'e de Montr\'eal, Montr\'eal, QC, Canada}\\
\tt \footnotesize{\{vinet,zhanhanm\}@crm.umontreal.ca}
}

\date{\today}
\begin{document}
\maketitle
\begin{abstract}
Fractional revival occurs between two vertices in a graph if a continuous-time quantum walk 
unitarily maps the characteristic vector of one vertex to a superposition of the characteristic 
vectors of the two vertices.
This phenomenon is relevant in quantum information in particular for entanglement generation in spin networks.
We study fractional revival in graphs whose adjacency matrices belong to the Bose-Mesner algebra 
of association schemes.
A specific focus is a characterization of balanced fractional revival 
(which corresponds to maximal entanglement) in graphs that belong to the Hamming scheme.
Our proofs exploit the intimate connections between algebraic combinatorics and orthogonal polynomials.

\end{abstract}



\section{Introduction}

Quantum walk on graphs is a fundamental area in quantum information and computation.
In quantum computation, it provides a natural generalization of Grover's celebrated algorithm to arbitrary graphs (see \cite{nc00}). 
In quantum information, it is important for studying transport problems in quantum spin networks. This was initiated by Bose \cite{b03} in the context of perfect state transfer in quantum spin chains.

A quantum transport problem that is relevant for entanglement generation is fractional revival. 
It is known that entanglement is a useful resource in quantum information theory with many
applications (for example, the teleportation protocol).
Fractional revival is also interesting since it captures both aspects of perfect state transfer and periodicity which are two well-known quantum transport phenomena (see Godsil \cite{g12}). 

Prior to our work,
Genest \etal \cite{gvz16} and Christandl \etal \cite{cvz17} had analytically
 studied the fractional revival phenomenon in quantum spin chains. The graphs they studied may be viewed as weighted paths with (possibly) additional edges connecting vertices at distance two from each other. 
 Based on techniques from orthogonal polynomials, they observed spectral conditions for fractional revival to occur in these weighted graphs.

In this work, we study quantum fractional revival mainly on unweighted graphs.
Our motivation is to understand the role of the underlying graph structure on fractional
revival without the benefit of arbitrary real-valued weights. 
To this end, we study fractional revival from a graph-theoretic perspective and 
develop some algebraic machinery useful for analyzing this phenomenon.
Other works with a similar focus include Bernard \etal \cite{bcltv18} 
and our prior work \cite{cctvz}.

The main combinatorial object we focus on is an association scheme. An association scheme is a set of matrices that satisfy strong regularity relations, which allow for a fairly combinatorial treatment of their spectral properties. More details will be presented in Section \ref{sec:assoc}. We present a full characterization of when fractional revival occurs in a graph that belongs to an association scheme --- a result that provides a way of easily and efficiently checking whether or when it occurs. Additionally, fractional revival in association schemes always features a partition of the vertex set of the graph into pairs of vertices, all of which exhibit fractional revival at the same time.

Following our general description of fractional revival in association schemes, we make a connection to distance-regular graphs and orthogonal polynomials, which in turn leads to our treatment of fractional revival in graphs whose adjacency matrices belong to a special association scheme --- the binary Hamming scheme $\HH(n,2)$. Our main result is a necessary and sufficient condition for balanced fractional revival to occur in this scheme. 
We then use this to provide constructions of explicit families of graphs with balanced fractional revival.
Balanced fractional revival is a natural choice since it corresponds to the generation of maximally entangled states (which are crucial for many quantum information theory protocols).


\section{Preliminaries}

We review some background from algebraic graph theory (see Godsil and Royle \cite{gr01}).
A graph $X$ is given by a set of its vertices $V(X)$ and a set of it edges $E(X)$.
The adjacency matrix of $X$, which we denote by $A(X)$, is a $01$ matrix whose $ab$ entry
is $1$ if $ab \in E(X)$ and is $0$ otherwise.
For a vertex $a \in V(X)$ where $|V(X)| = n$, 
we use $\ee_{a}$ to denote the unit (characteristic) vector of dimension $n$
that is $1$ at position indexed by $a$ and is $0$ elsewhere.

For a graph $X$, the continuous-time quantum walk (or transition) matrix of $X$ is given by
\begin{equation}
U(t) = e^{-\ii t A(X)}
\end{equation}
where $t$ ranges over the reals. 
This was originally studied by Farhi and Gutmann \cite{fg98} in the context of decision trees.

Let $a$ and $b$ be two distinct vertices of $X$. 
We say that $X$ admits {\em fractional revival from $a$ to $b$} at time $\tau > 0$
if for some $\alpha, \beta \in \CC$, with $\beta \neq 0$ and $|\alpha|^2 + |\beta|^2 = 1$, we have
\begin{align} \label{eqn:fr_def}
U(\tau) \ee_a = \alpha \ee_a + \beta \ee_b. 
\end{align}
In this case, we also say {\em $(\alpha,\beta)$-revival occurs from $a$ to $b$ at time $\tau$}.
By factoring a common unimodular phase factor, we may assume $\alpha$ is real.
So, we say $e^{\ii\zeta}(\alpha,\beta)$-revival occurs where $\alpha$ and $\zeta$
are real scalars and $\beta$ is complex.
The fractional revival is called {\em balanced} if $|\alpha| = |\beta| = 1/\sqrt{2}$.
In this case, we may simply say balanced fractional revival occurs with phase $\zeta$.

We say $X$ has $(\alpha,\beta)$-revival if it has $(\alpha,\beta)$-revival from {\em every} vertex 
at the same time. This holds if there is a permutation matrix $T$ 
(with no fixed points) where for some time $\tau$ we have
\begin{equation}
U(\tau) = \alpha I + \beta T.
\end{equation}

We define several other quantum transport properties.
The graph $X$ is called {\em periodic at vertex $a$} at time $\tau$ 
if $\beta = 0$ in \eqref{eqn:fr_def}.
We say $X$ has {\em perfect state transfer} from $a$ to $b$ at time $\tau$
if $\alpha = 0$ in \eqref{eqn:fr_def}.
See Godsil \cite{g12} for a survey of these notions.


\section{Association Schemes} \label{sec:assoc}

A {\em symmetric association scheme} is a set of $n\times n$ symmetric $01$-matrices $\{A_0,...,A_d\}$ satisfying the following properties:
\begin{enumerate}[(i)]
	\item $\sum_{i = 0}^d A_i = J$, the all $1$s matrix.
	\item The identity matrix is one of them, and we always assume $A_0 = I$.
	\item $A_i A_j$ is a linear combination of the matrices in the scheme, for all $i$ and $j$.
\end{enumerate}
Each matrix in the association scheme is called a {\em class of the scheme}. As the classes are symmetric, property (iii) implies that they all commute. By taking products and linear combinations, it follows that the matrices $\{A_0,...,A_d\}$ generate a commutative algebra of matrices, called the Bose-Mesner algebra of the scheme, which we denote by $\cal A$. The association scheme always forms a basis for its Bose-Mesner algebra. This is an algebra of symmetric commuting matrices, thus it can be simultaneously diagonalized. The projectors onto the eigenspaces are also a basis for the algebra, therefore there are $d+1$ of those, which we typically denote by $\{E_0,...,E_d\}$. The Bose-Mesner algebra is closed under ordinary matrix product, and $\{E_0,...,E_d\}$ are the minimal idempotents. It is also closed under the entrywise product of matrices, and the association scheme $\{A_0,...,A_d\}$ is a basis of minimal idempotents for this product. Two square $(d+1)\times (d+1)$ matrices $P$ and $Q$ are typically defined in order to record the relationship between these two bases, as follows:
\begin{itemize}
    \item for all $i$, $A_i = \sum_{j = 0}^d P_{ji} E_j$; and
    \item for all $j$, $E_j = \frac{1}{n} \sum_{i = 0}^d Q_{ij} A_i$. 
\end{itemize}
Note that $PQ = nI$. We refer the reader to \cite{bcn89} for a complete introduction to association schemes, and to \cite{cggv15} for an application of the concept related to quantum walks.

For the result that follows, let $X$ be a graph that belongs to an association scheme, which is to say, $A = A(X)$ is a sum of some of the matrices in an association scheme. Thus $A$ belongs to the Bose-Mesner algebra, and so does $U(t)$, for any $t$. Hence it is a linear combination of matrices in the scheme, and this shall be sufficient for a fairly restricted description of when and how fractional revival can occur.

\begin{theorem} \label{thm:fr_in_scheme} 
Let $X$ be a graph with $A=A(X)$ in the Bose-Mesner algebra of an association scheme $\{A_0,A_1,\ldots,A_d\}$ 
with minimal idempotents $\{E_0,E_1,\ldots,E_d\}$. 
Suppose the eigenvalues of $A$ are given by $\theta_0 \ge \theta_1 \ge \ldots \ge \theta_d$. 

There are real scalars $\alpha,\zeta$ and a complex scalar $\beta$ so that the graph $X$ has $e^{\ii\zeta}(\alpha,\beta)$ fractional revival from $a$ to $b$ at time $\tau$ if and only if both the following conditions hold.
\begin{enumerate}[(a)]
\item the unique class $A_{q}$ of the scheme which is non-zero in the $(a,b)$ entry is a permutation matrix of order $2$ (and so its eigenvalues are $\pm 1$); and \label{cond:1}

\item for all $r \in \{1,...,d\}$, if $A_q E_r = E_r$, then $(\theta_r - \theta_0) \tau \equiv 0 \pmod{2\pi}$, and if  $A_q E_r = -E_r$, then
\begin{equation} \label{eqn:fr_in_scheme}
(\theta_r - \theta_0)\tau 
	\equiv
	2\cos^{-1}(\alpha) \pmod{2\pi}
\end{equation} \label{cond:2}
\end{enumerate}
Note that if the conditions hold, then fractional revival occurs between all pairs of vertices determined by $A_q$, as in fact,
\begin{align}U(\tau) = e^{\ii \zeta} (\alpha I + \beta A_q). \label{eq:1}\end{align}

\begin{proof} 
We first show Condition \eqref{cond:1} is necessary. Assume $X$ has $e^{\ii\zeta}(\alpha,\beta)$ fractional revival from $a$ to $b$ at time $\tau$, where $\alpha \in \RR$. 
Thus, $U(\tau)\ee_{a} = e^{\ii\zeta}(\alpha\ee_{a} + \beta\ee_{b})$.  
Since $U(\tau)$ belongs to the Bose-Mesner algebra of the scheme,
we have $U(\tau) = \sum_{r = 0}^d \eta_{r}A_{r}$ for some constants $\eta_{r}$ (which depend on $\tau$). Suppose $A_{q}$ is the unique matrix in the scheme for which $(A_{q})_{b,a} = 1$. 
Thus, $\eta_{0} = e^{\ii\zeta}\alpha$, $\eta_{q} = e^{\ii\zeta}\beta$, 
and $\eta_{\ell} = 0$ for $\ell \neq 0,q$.  
This shows that 
\begin{equation} \label{eqn:local2global}
U(\tau) = e^{\ii\zeta}(\alpha A_{0} + \beta A_{q}).  
\end{equation}
As $A_{q}$ commutes with $J$, it is a symmetric matrix with row and column sums equal $1$.  
Hence $A_{q}$ is a permutation matrix of order $2$.

Now we assume Condition \eqref{cond:1} holds. Fractional revival between $a$ and $b$ is equivalent to Equation \eqref{eq:1}, which we now show to be equivalent to Condition \eqref{cond:2}.

Let $A_q = \sum_{r = 0}^d \sigma_r E_r$. Again, observe that $\sigma_r = \pm 1$. We have
\begin{equation} \label{eqn:local3global}
U(\tau) = e^{\ii\zeta}(\alpha A_{0} + \beta A_{q}).  
\end{equation}
if and only if, for all $r \in \{0,...,d\}$,
\[e^{-\ii \theta_r \tau} = e^{\ii\zeta}(\alpha + \beta \sigma_r).\]
This is true because the idempotents $\{E_0,...,E_r\}$ form a basis. Now, as $|e^{\ii\zeta}(\alpha \pm \beta)| = 1$ and $\alpha \in \R$, we have $\beta \in \ii \R$, and since $|\alpha|^{2} + |\beta|^{2} = 1$, we may assume $\alpha = \cos \vartheta$ and $\beta = \ii \sin \vartheta$. Hence Equation \eqref{eqn:local3global} is equivalent to, for all $r$, 
\[e^{-\ii \theta_r \tau} = e^{\ii\zeta} e^{\sigma_r \ii  \vartheta}.\]
Fixing $-\theta_0 \tau = \zeta + \vartheta$ (as $\sigma_0 = 1$), the equation above, for each $r \neq 0$, is equivalent to
\begin{enumerate}[(i)]
\item if $\sigma_r = 1$, then $(\theta_r - \theta_0) \tau \equiv 0 \pmod{2\pi}$; and
\item if $\sigma_r = -1$, then $(\theta_r - \theta_0)\tau 
	\equiv	2\vartheta \pmod{2\pi}$,
\end{enumerate}
which is precisely Condition \eqref{cond:2} of the statement.
\end{proof} 
\end{theorem}

We can say a bit more. If there is fractional revival between two vertices in a graph in an association scheme, the result above says that there is a permutation matrix in the scheme that swaps these two vertices. Thus they are strongly cospectral (see \cite[Theorem 11.2]{godsilsmith}). As a consequence, we can apply \cite[Corollary 5.6]{cctvz}, and because $\theta_0$ is always an integer if $X$ is regular, we have the corollary below.

\begin{corollary}\label{cor:integral}
If fractional revival occurs in a graph $X$ belonging to an association scheme at time $\tau$, then all eigenvalues of $X$ are integers, and $\tau$ is a rational multiple of $\pi$.
\qed
\end{corollary}

This takes us immediately to the following characterization.

\begin{theorem}\label{thm:frchar}
Let $\{A_0,...,A_d\}$ be an association scheme whose set of minimal idempotents is given by $\{E_0,...,E_d\}$. Assume one of its classes is a permutation matrix 
of order two, say $A_q$ for some $q$, with spectral decomposition $A_q = \sum_{r = 0}^d \sigma_s E_s$. Let $X$ be a graph in this scheme, with integer eigenvalues $\theta_0 \geq ... \geq \theta_d$. Define
\[g = \gcd\big\{(\theta_0 - \theta_s)\big\}_{s = 1}^d.\]
Then for all integers $m \ge 1$ satisfying the properties
\begin{enumerate}[(i)]
    \item $\{(\theta_0 - \theta_s)/g\}_{\sigma_s = -1}$ are all congruent to the same integer $\mu$ modulo $m$, and
    \item $\{(\theta_0 - \theta_s)/g\}_{\sigma_s = 1}$ are all congruent to $0$ modulo $m$,
\end{enumerate}
it follows that $X$ admits $e^{\ii \zeta}(\alpha,\beta)$ fractional revival at time $\tau = 2 t \pi /m g$, for any integer $t$, where
\[\zeta = t\pi \left(\frac{-2\theta_0 - \mu g}{m g}\right)\quad , \quad \alpha = \cos \left( \frac{t\mu \pi}{m}\right) \quad \text{and} \quad \beta = \ii \sin \left(\frac{t\mu \pi}{m} \right) \]
Moreover, if fractional revival occurs at time $\tau$, then there are integers $m$ and $\mu$ such that (i) and (ii) above hold.
\end{theorem}

Note that if $\mu \neq 0$, then $g \cdot \gcd(\mu,m)$ divides $\theta_0-\theta_s$ for $s=1,\ldots, d$, so we either have $\gcd(\mu,m)=1$ or $\mu=0$.
In particular, if $\mu = 0$ in the theorem, it describes periodicity (and it always happens choosing the degenerate case $m = 1$ as, after all, the eigenvalues are integers). If $\mu = 1$ and $m = 2$, the theorem describes perfect state transfer. For any other $m$, as long as $\mu \neq 0$, we have fractional revival with both $\alpha$ and $\beta$ different than $0$. If $m=4$ and $\mu = \pm 1$, we have balanced fractional revival.

\begin{proof}
If $m$ and $\mu$ exist satisfying properties $(i)$ and $(ii)$, then clearly Condition \eqref{cond:2} of Theorem \ref{thm:fr_in_scheme} is satisfied at $\tau$, and parameters $\zeta$ (modulo $2 \pi$), $\alpha$ and $\beta$ are determined. Now assume fractional revival occurs at a time $\tau$, which, by Corollary \ref{cor:integral}, we may assume satisfies $\tau = (a/b) \pi$ for some (positive) integers $a$ and $b$. We may further assume $a < 2b$, because a graph with integer eigenvalues is periodic at time $2 \pi$. Let $m$ be the smallest integer so that for some integer $t$, the equality below holds
\[\frac{a}{b} \pi = \frac{2t}{mg} \pi.\]
We must show now that conditions (i) and (ii) hold for some $\mu$. From Condition \eqref{cond:2} of Theorem \ref{thm:fr_in_scheme}, for all $s$ with $\sigma_s = 1$, we have
\[(\theta_0 - \theta_s)\frac{2t}{mg} \pi \equiv 0 \pmod{2 \pi}, \]
and as $t$ and $m$ are coprime, this is equivalent to $(\theta_0 - \theta_r)/g \equiv 0 \pmod{m}$. Likewise, for all $s$ with $\sigma_s = -1$, we have, for some $\vartheta$,
\[(\theta_0 - \theta_s)\frac{2t}{mg} \pi \equiv 2 \vartheta \pmod{2 \pi}, \]
and as $t$ and $m$ are coprime, this is equivalent to the existence of an integer $\mu$ with $(\theta_0 - \theta_s)/g \equiv \mu \pmod{m}$.
\end{proof}

The conditions in Theorem \ref{thm:frchar} are quite descriptive, but we can do better in terms of providing an efficient way of checking whether fractional revival (or its variants) occur.

\begin{lemma}\label{lem:frchar}
Let $\{A_0,...,A_d\}$ be an association scheme whose set of minimal idempotents is given by $\{E_0,...,E_d\}$. Assume one of its classes is a permutation matrix 
of order $2$, say $A_q$ for some $q$, with spectral decomposition $A_q = \sum_{r = 0}^d \sigma_s E_s$. Let $X$ be a graph in this scheme, with integer eigenvalues $\theta_0 \geq ... \geq \theta_d$. Define
\[g = \gcd \big\{(\theta_0 - \theta_s)\big\}_{s = 1}^d.\]
Let $h$ be the integer satisfying
\[hg = \gcd\big\{(\theta_r - \theta_s)\big\}_{\sigma_r = \sigma_s}.\]
Then a positive integer $m$ satisfies Theorem~\ref{thm:frchar} if and only if it divides $h$.
\end{lemma}
\begin{proof}
The integer $m$ satisfies the properties of Theorem \ref{thm:frchar} if and only if, for $\sigma_s=\sigma_r$,
\[\frac{\theta_r - \theta_s}{g} = \frac{\theta_0 - \theta_s}{g} - \frac{\theta_0- \theta_r}{g} \equiv 0 \pmod m,\]
the last equivalence being true because either $0-0 \equiv 0 \pmod m$ or $\mu - \mu \equiv 0 \pmod m$.   This equation holds exactly when $m$ divides $h$.
\end{proof}

\begin{theorem}\label{thm:frchar_gh}
Let $X$ be a graph in an association scheme containing a permutation matrix of order two.   Let $g$ and $h$ be defined as in Lemma~\ref{lem:frchar}.
Then
\begin{enumerate}[(a)]
\item $h = 1$ if and only if the graph $X$ does not admit fractional revival nor perfect state transfer.

\item $h > 2$ if and only if the graph $X$ admits fractional revival that is different from perfect state transfer.  In particular, $\frac{2\pi}{hg}$ is the minimum time when fractional revival occurs in $X$ 

\item $h$ is even if and only if the graph $X$ admits perfect state transfer.  In particular, when $h=2$, $X$ admits perfect state transfer but no other form of fractional revival.
\item $h$ is doubly even if and only if the graph $X$ admits balanced fractional revival.
\qed
\end{enumerate}
\end{theorem}

\subsection{Distance-Regular Graphs} 

Association schemes can be constructed in several distinct ways. One of them comes from certain graphs and their distance matrices. If $X$ is a graph, let $A_k$ be the $01$ symmetric matrix indexed by vertices with $(i,j)$ entry equal to $1$ if and only if vertex $i$ is at distance $k$ from vertex $j$. Note that $A_0 = I$ and $A_1$ is simply the adjacency matrix. To a graph of diameter $d$ we can associate the set $\{A_0,...,A_d\}$ of its distance matrices. Now in very few special cases, the set of distance matrices of a graph form an association scheme. When this happens, the graph is called {\em distance-regular}. See \cite{bcn89} for more background on distance-regular graphs.

A distance-regular graph $X$ is called {\em primitive} if each of the graphs $X_{k}$ whose adjacency matrix coincide with the $k$-th distance matrix of $X$, $A_k$, are connected. If any of them is disconnected, $X$ is called {\em imprimitive}. If $X$ is imprimitive and $k$-regular with $k \ge 3$, then $X$ is bipartite or $X_{d}$ is a disjoint union of $\ell$-cliques, for some $\ell$.  In the latter case, $X$ is called {\em antipodal}, the cliques in $X_{d}$ are called the fibres of $X$, and the fibre size is $\ell$. (See Theorem 4.2.1 in \cite{bcn89}.)

\begin{proposition} \label{prop:fr_in_drg} 
Let $X$ be a distance-regular graph of diameter $d > 1$. If $X$ has $(\alpha,\beta)$ fractional revival from $a$ to $b$ at time $\tau$, then 
\begin{equation} 
U_{X}(\tau) = \alpha A_{0} + \beta A_{d}, 
\end{equation} 
where $A_{d}$ is the adjacency matrix of $\frac{d}{2}K_{2}$.  

\begin{proof} 
By Theorem \ref{thm:fr_in_scheme}, we know that $U_{X}(\tau) = \alpha A_{0} + \beta A_{q}$ for some $1 \le q \le d$, and that $A_q$ is a permutation matrix of order two with no fixed points. By reasoning exactly as in \cite[Theorem 4.1]{cggv15}, it must be that $q = d$.  
\end{proof}
\end{proposition}

If $X$ is a distance-regular graph with corresponding distance matrices $\{A_0,...,A_d\}$, it is not hard to see that there are polynomials $p_0,...,p_d$, with $p_k$ having degree $k$, so that $p_k(A_1) = A_k$. These polynomials are all orthogonal according to a standard choice of inner product, 
because the Schur product between $A_i$ and $A_j$ vanishes, that is, 
$A_i \circ A_j = 0$, for $i \neq j$, and thus $\tr A_i A_j = 0$. If $X$ is antipodal with eigenvalues $\theta_0 > ... > \theta_d$ and corresponding projectors $\{E_0,...,E_d\}$, and $A_d$ is a permutation matrix of order $2$, then \cite[Proposition 11.6.2]{brouwerhaemers} tells us that
\[A_d = \sum_{r = 0}^d (-1)^r E_r.\]
Along with the characterization in Theorem~\ref{thm:frchar_gh}, this provides a very efficient and easy method to check whether a distance-regular graph (DRG) admits fractional revival (FR). We use PST to abbreviate perfect state transfer in the table below.  We summarize the results for the families of distance-regular graphs studied in \cite{cggv15}.
\small{
\begin{flushleft}
\renewcommand{\arraystretch}{1.5}
\begin{tabular}{|l|l|c|c|c|l|}
\hline
 DRG & spectrum&$g$ & $h$ & $\mu$ & FR\\
\hline
  $\overline{nK_2}$ & $\{2n-2,0,-2\}$&$2$ & $n$ & $n-1$ & FR at $\frac{k\pi}{n}$, $k=1,\ldots,n-1$\\
  &&&&& PST at $\frac{\pi}{2}$ iff $n$ is even\\
  \hline
 2-fold cover of $K_n$& $\{n-1, \frac{\delta+\sqrt{\Delta}}{2}, -1, \frac{\delta-\sqrt{\Delta}}{2}\}$&&&&\\
\ $\delta=0$ &&2&1&- & no FR, no PST\\
\ $\delta=-2$&&$\sqrt{n}$&$2$&$1$& no FR, PST at $\frac{\pi}{\sqrt{n}}$\\
\ $\delta=2$,\ $n\equiv 4 \pmod{8}$& &$4$ & $\frac{\sqrt{n}}{2}$& $\frac{\sqrt{n}-2}{4}$& FR at $\frac{\pi}{\sqrt{n}}$, no PST\\
\ $\delta=2$, \ $n\equiv 0 \pmod{8}$& & $2$ & $\sqrt{n}$& $\frac{\sqrt{n}-2}{2}$&  FR at $\frac{\pi}{\sqrt{n}}$, PST at $\frac{\pi}{2}$ \\
\hline
Hadamard graph & $\{n^2,n,0,-n,-n^2\}$ & &&&\\
\quad of order $n^2$ && $n$& 2&1& no FR, PST at $\frac{\pi}{n}$\\
\hline
$n$-cube & $\{n-2j\}_{j=0}^n$ & 2 & 2 & 1 & no FR, PST at $\frac{\pi}{2}$\\
\hline
halved $2d$-cube & $\{\binom{2d}{2} - 2j(2d-j)\}_{j=0}^d$ & 2 & 4 & $2d-1$& balanced FR at $\frac{\pi}{4}$\\ 
&&&&& PST at $\frac{\pi}{2}$\\
\hline
Johnson graph $J(2n,n)$ & $\{(n-j)^2-j\}_{j=0}^n$ & 2 & 1& -& no FR, no PST\\
\hline
Doubled odd graph  & $(-1)^j(n+1-j)$,& 1&1&-&no FR, no PST\\ 
\quad on $(2n+1)$ points & \quad ($j\neq n+1$) &&&&\\
\hline
\end{tabular}
\renewcommand{\arraystretch}{1}
\end{flushleft}}


\section{Hamming Scheme}

In this section, we focus on the Hamming scheme and provide characterization when balanced
fractional revival occurs.

Consider families of graphs whose vertices are the binary sequences of length $n$, where $n \ge 1$.
The graph $X_{r}$, for $r=0,\ldots,n$, has edges connecting all pairs of vertices with
Hamming distance $r$. The graph $X_1$ is also known as the $n$-cube.  Let $A_{r} = A(X_{r})$ be the adjacency matrix of $X_{r}$.
This describes the well-known Hamming scheme $\HH(n,2)$.
So, the Bose-Mesner algebra of $\HH(n,2)$ is spanned by the set of
matrices $\mathcal{A} = \{A_{0},A_{1},\ldots,A_{n}\}$.
Let $\{E_0,\ldots,E_n\}$ be the set of minimal idempotents of this scheme, where
\[A_1 E_s= (n-2s)E_s
\quad \text{and} \quad
A_n E_s = (-1)^s E_s, \quad \text{for $s=0,\ldots,n$.}\]
Then, we have (see Stanton \cite{s01}, Section 2)
\begin{equation}
A_{r} = \sum_{s=0}^{n} p_{r}(s)E_{s},
\ \hspace{.5in} \
r = 0,\ldots,n.
\end{equation}
where $p_{r}(x)$ is the Krawtchouk polynomial of degree $r$.
A more customary notation for the Krawtchouk polynomial is $p_r(x,q,n)$ which specifies the arity $q$ and the dimension $n$. 
We suppress $q$ since we focus exclusively on $q=2$ and we will include $n$ only if necessary.

It is known that
(see Stanton \cite{s01}, Equation (2.3a))
\begin{equation}
p_{r}(s) = \binom{n}{r} 
	{}_{2}F_{1}\left( \begin{tabular}{c|} $-r,-s$ \\ $-n$ \end{tabular} \ 2 \right),
\ \mbox{ where } \ 
{}_{2}F_{1}\left( \begin{tabular}{c|} $a,b$ \\ $c$ \end{tabular} \ x \right) =
	\sum_{m=0}^{\infty} \frac{a^{\overline{m}}b^{\overline{m}}}{c^{\overline{m}}} \frac{x^{m}}{m!}.
\end{equation}
Here, ${}_{2}F_{1}$ is the Gaussian hypergeometric function
and $z^{\overline{m}} = z(z+1) \ldots (z+m-1)$ is the $m$-th rising factorial power of $z$.

\medskip
We state some useful properties of the Krawtchouk polynomials.

\begin{proposition} \label{prop:krawtchouk}
For $n \ge 2$, the Krawtchouk polynomials $p_{r}(s)$ satisfy:
\begin{enumerate}[(i)]
\item (MacWilliams and Sloane \cite{ms77}, Chapter 5, Section 7, Theorem 15) \\
For $0 \le r \le n$,
\begin{equation} \label{eqn:krawtchouk}
p_{r}(s) = \sum_{h=0}^{r} (-2)^{h} \binom{n-h}{r-h} \binom{s}{h}
\ \hspace{.5in} \
s = 0,\ldots,n.
\end{equation}
Therefore,
\begin{equation}
p_{r}(1) - p_{r}(0) = -2\binom{n-1}{r-1}  
\end{equation}.

\item (Chihara and Stanton \cite{cs90}, Proposition 2.1) \\
For $1\le r\le n$, 
\begin{equation}\label{eqn_rec_1}
    p_r(s, n) - p_r(s+1, n) = 2p_{r-1}(s, n-1),
    \ \hspace{0.5in} \ 
s = 0, \ldots, n-1.
\end{equation}

\item (Chihara and Stanton \cite{cs90}, Proposition 2.3) \\
For $1 \le r \le n$, 
\begin{equation}\label{eqn_rec_2}
p_{r}(s,n) - p_{r}(s+2,n) = 4p_{r-1}(s,n-2),
\ \hspace{0.5in} \ 
s = 0, \ldots, n-2.
\end{equation}
Therefore,
\begin{equation}
p_{r}(s) - p_{r}(s+2) = 4\sum_{h=0}^{r-1} (-2)^{h}\binom{n-2-h}{r-1-h} \binom{s}{h},
\ \hspace{0.5in} \ 
s = 0, \ldots, n-2.
\qed
\end{equation}
\end{enumerate}
\end{proposition}

From Theorem~\ref{thm:frchar_gh}, we see that for a graph $X$ with eigenvalue $\theta_0,\ldots,\theta_r$, the parameters
\[\gcd \big\{(\theta_0-\theta_s)\big\}_{s=1}^n\]
and
\[\gcd \big\{(\theta_r - \theta_s)\big\}_{\sigma_r = \sigma_s}.\]
play a role in characterizing fractional revival. If, in addition, $X$ is a class of the binary Hamming scheme $\HH(n,2)$, then we can express these two parameters using the eigenvalues of the scheme.

To simplify the computation in the following proof, we define $p_r(s,n)=0$ if $r<0$ or $r>n$.
We can then extend Equations~(\ref{eqn_rec_1}) and (\ref{eqn_rec_2}) to any integer $r$.
\begin{lemma}\label{lem:gcd}
Let $A=A_{r_1}+ \cdots + A_{r_{\ell}}$, where $0 < r_1 < \cdots < r_{\ell}\le n$, in $\HH(n,2)$. Suppose $A \neq A_{n}$ and $A=\sum_{s=0}^n \theta_s E_s$, where $E_0,\ldots, E_n$ are the minimal idempotents of $\HH(n,2)$.
Define
\[g = \gcd \big\{(\theta_0-\theta_s)\big\}_{s=1}^n,\]
and $h$ to be the integer satisfying
\[hg = \gcd\big\{(\theta_s-\theta_{s+2})\big\}_{s=0}^{n-2}.\]
Then
\[g = \gcd\Big\{2^j  \sum_{i=1}^{\ell} p_{r_i-j}(0, n-j)\Big\}_{j=1}^{r_{\ell}},\]
and
\[hg =
 \gcd\Big\{2^{j+1} \sum_{i=1}^{\ell'} p_{r_i-j}(0,n-j-1)\Big\}_{j=1}^{r_{\ell'}},
 \quad
\text{where} 
\  
 \ell' = \begin{cases} \ell & \text{if $r_{\ell}<n$,}\\ \ell-1 & \text{if $r_{\ell}=n$.}\end{cases}\]
In particular, $g$ is even and divides $2^{r_{\ell}}$, and $hg$ is doubly even and divides $2^{r_{\ell'}+1}$.
\end{lemma}
\begin{proof}
First we have
\[g = \gcd\Big\{\theta_s - \theta_{s+1}\Big\}_{s=0}^{n-1} 
= \gcd\Big\{\sum_{i=1}^{\ell} \big(p_{r_i}(s,n)-p_{r_i}(s+1,n)\big)\Big\}_{s=0}^{n-1}.\]
Applying Equation~(\ref{eqn_rec_1}) gives
\begin{eqnarray*}
g&=&\gcd\Bigg\{2\sum_{i=1}^{\ell}p_{r_i-1}(0,n-1),2\sum_{i=1}^{\ell}p_{r_i-1}(1,n-1),\ldots,2\sum_{i=1}^{\ell}p_{r_i-1}(n-1,n-1)\Bigg\}\\
&=& \gcd\Bigg\{2\sum_{i=1}^{\ell}p_{r_i-1}(0,n-1), 
\gcd\Big\{2\sum_{i=1}^{\ell} \big(p_{r_i-1}(s,n-1)-p_{r_i-1}(s+1,n-1)\big)\Big\}_{s=0}^{n-2}\Bigg\}.
\end{eqnarray*}
Applying Equation~(\ref{eqn_rec_1}) repeatedly gives
\begin{eqnarray*}
g &=& \gcd\Bigg\{2\sum_{i=1}^{\ell}p_{r_i-1}(0,n-1), 
\gcd\Big\{2^2\sum_{i=1}^{\ell} p_{r_i-2}(s,n-2)\Big\}_{s=0}^{n-2}\Bigg\}\\
&=&\gcd\Bigg\{2\sum_{i=1}^{\ell}p_{r_i-1}(0,n-1), 2^2\sum_{i=1}^{\ell} p_{r_i-2}(0,n-2),\gcd\Big\{2^2\sum_{i=1}^{\ell} \big(p_{r_i-2}(s,n-2)-p_{r_i-2}(s+1,n-2)\big)\Big\}_{s=0}^{n-3}\Bigg\}\\
&&\\
&=& \cdots\\
&&\\
&=& \gcd\Bigg\{2 \sum_{i=1}^{\ell}p_{r_i-1}(0,n-1),
    2^2 \sum_{i=1}^{\ell}p_{r_i-2}(0,n-2),\cdots,2^{r_{\ell}-1}\sum_{i=1}^{\ell}p_{r_i-r_{\ell}+1}(0,n-r_{\ell}+1),\\
    && \qquad 
     2^{r_{\ell}} \gcd\Big\{\sum_{i=1}^{\ell}p_{r_i-r_{\ell}}(s,n-r_{\ell})\Big\}_{s=0}^{n-r_{\ell}}\Bigg\}.
\end{eqnarray*}
Since $p_{r_i-r_{\ell}}(s,n-r_{\ell})=0$ for $i=1,\ldots, \ell-1$, and $p_{r_{\ell}-r_{\ell}}(s,n-r_{\ell})=p_0(s,n-r_{\ell})=1$ for $s=0,1,\cdots,n-r_{\ell}$, the last equation reduces to 
\[g=\gcd\Bigg\{2 \sum_{i=1}^{\ell}p_{r_i-1}(0,n-1),
    2^2 \sum_{i=1}^{\ell}p_{r_i-2}(0,n-2),\cdots,2^{r_{\ell}-1}\sum_{i=1}^{\ell}p_{r_i-r_{\ell}+1}(0,n-r_{\ell}+1), 2^{r_{\ell}}\Bigg\}.\]
In particular, $2$ divides $g$ and $g$ divides $2^{r_{\ell}}$.

To compute $hg$, first observe that $p_n(s,n)=p_n(s+2,n)$, for $s=0,\ldots,n-2$.
Hence 
\[
\theta_s-\theta_{s+2} = \sum_{i=1}^{\ell'} \big(p_{r_i}(s,n) - p_{r_i}(s+2,n) \big),
\quad \text{where}\ \ell' = \begin{cases} \ell & \text{if $r_{\ell}<n$,}\\ \ell-1 & \text{if $r_{\ell}=n$.}\end{cases}\]
By Equation (\ref{eqn_rec_2}), we have
\begin{eqnarray*}
hg &=& \gcd\Bigg\{4\sum_{i=1}^{\ell'}p_{r_i-1}(0,n-2),4\sum_{i=1}^{\ell'}p_{r_i-1}(1,n-2),\cdots, 4\sum_{i=1}^{\ell'}p_{r_i-1}(n-2,n-2)\Bigg\}\\
&=&
\gcd\Bigg\{2^2\sum_{i=1}^{\ell'}p_{r_i-1}(0,n-2), \gcd\Big\{2^2\sum_{i=1}^{\ell'}\big(p_{r_i-1}(s,n-2)-p_{r_i-1}(s+1,n-2)\big)\Big\}_{s=0}^{n-3}\Bigg\}.
\\
\end{eqnarray*}
Similar to the computation of $g$, we apply Equation (\ref{eqn_rec_1}) repeatedly to get
\[hg = \gcd\Big\{2^2 \sum_{i=1}^{\ell'}p_{r_i-1}(0,n-2),2^3 \sum_{i=1}^{\ell'} p_{r_i-2}(0,n-3),\cdots,
2^{r_{\ell'}}\sum_{i=1}^{\ell'}p_{r_i-r_{\ell'}+1}(0,n-r_{\ell'}),2^{r_{\ell'}+1}\Big\}.\]
In particular, $4$ divides $hg$ and $hg$ divides $2^{r_{\ell'}+1}$.
\end{proof}

For a positive integer $m$, let $\alpha_2(m)$ denote the largest $k$ such that $2^k$ divides $m$. We now give formulas for $g$ and $hg$ in terms of binomial coefficients.

\begin{proposition}
Let $A=A_{r_1}+ \cdots + A_{r_{\ell}}$, where $0 < r_1 < \cdots < r_{\ell}\le n$, in $\HH(n,2)$.
If $A\neq A_n$ then 
\[\log_2(g) = \min\left\{j + \alpha_2 (\sum_{i=1}^{\ell}\binom{n-j}{r_i-j}): j=1,2,\cdots,r_{\ell}\right\},\]
and
\[\log_2(hg) = \min\left\{j + 1 + \alpha_2(\sum_{i=1}^{\ell'}\binom{n-j-1}{r_i-j}): j=1,2,\cdots,r_{\ell'}\right\},\]
where
\[
\ell'=
\begin{cases}
\ell & \text{if $r_{\ell}<n$,}\\
\ell-1 & \text{if $r_{\ell}=n$.}
\end{cases}
\]
\end{proposition}
\begin{proof}
The result follows since $g$ and $hg$ are powers of $2$
and $p_r(0,n)= \binom{n}{r}$, for $r=0,1,\ldots,n$.
\end{proof}
Using Theorem~\ref{thm:frchar_gh}, we get the following characterization of simple graphs in $\HH(n,2)$ that have fractional revival and the minimum time fractional revival can occur.

\begin{theorem}
Let $X=X_{r_1}\cup \cdots \cup X_{r_{\ell}}$, where $0< r_1< \cdots < r_{\ell}\le n$, in $\HH(n,2)$, and
\[
\ell'=
\begin{cases}
\ell & \text{if $r_{\ell}<n$,}\\
\ell-1 & \text{if $r_{\ell}=n$.}
\end{cases}
\]
Suppose $X\ne X_n$.  
Then $X$ has fractional revival if and only if 
\[\min\left\{j + \alpha_2(\sum_{i=1}^{\ell'} \binom{n-j-1}{r_i-j}): j=1,2,\cdots,r_{\ell'}\right \} \ge \min\left\{j + \alpha_2(\sum_{i=1}^{\ell}\binom{n-j}{r_i-j}): j=1,2,\cdots,r_{\ell}\right\}.\]
Moreover, if equality holds, then $X$ has perfect state transfer but no other form of fractional revival, and if strict inequality holds, then $X$ has perfect state transfer, as well as balanced fractional revival.
\qed
\end{theorem}

\begin{theorem}\label{thm:gfr_tight}
Let $X=X_{r_1}\cup \cdots \cup X_{r_{\ell}}$, where $0< r_1< \cdots < r_{\ell}\le n$, in $\HH(n,2)$, and
\[
\ell'=
\begin{cases}
\ell & \text{if $r_{\ell}<n$,}\\
\ell-1 & \text{if $r_{\ell}=n$.}
\end{cases}
\]
Suppose $X\ne X_n$. If $X$ has fractional revival at minimum time $\tau$, then $\tau = \pi/2^k$ for some positive integer $k\le r_{\ell'}$.
\end{theorem}
\begin{proof}
From Theorem~\ref{thm:frchar_gh}, $2\pi/hg$ is the minimum time when fractional revival occurs in $X$,
and by Lemma~\ref{lem:gcd}, $hg$ divides  $2^{r_{\ell'}+1}$.
\end{proof}


\subsection{The Importance of Being Balanced}

Balanced fractional revival is the most natural and relevant variant of fractional revival useful for generating maximally entangled states in quantum networks. In this section, we derive necessary and sufficient conditions for balanced fractional revival to occur in a union of graphs in $\HH(n,2)$.

\begin{proposition} \label{prop:fr_hamming}
Let $X=X_{r_1}\cup \cdots \cup X_{r_{\ell}}$, where $0< r_1< \cdots < r_{\ell}\le n$, in $\HH(n,2)$
Suppose $X\ne X_n$. 
Then $X$ has $e^{\ii\zeta}(\cos\pi/4, \pm\ii\sin\pi/4)$-revival at time $\pi/2^{k}$
if and only if the following conditions hold.
\begin{enumerate}[(i)]
\item \label{eqn:uno} 
	\begin{equation} 
	\sum_{i=1}^{\ell} \binom{n-1}{r_{i}-1} \equiv \pm 2^{k-2}\pmod{2^{k}}. 
	\end{equation}

\item \label{eqn:dos} For all $j = 1, \ldots, k-1$, 
	\begin{equation}
	\sum_{i=1}^{\ell} \binom{n-1-j}{r_{i}-j} \equiv 0 \pmod{2^{k-j}}.
	\end{equation}

\item \label{eqn:tres} The phase $e^{\ii \zeta}$ satisfies
	\begin{equation} 
	\zeta + \frac{\sum_{i=1}^{\ell} \binom{n}{r_{i}}\pi}{2^k} \equiv \mp \frac{\pi}{4} \pmod{2\pi}. 
	\end{equation}

\end{enumerate}
\end{proposition}
\begin{proof}
Let $g$ and $h$ be as defined in Lemma~\ref{lem:gcd}.   Then both $g$ and $hg$ are powers of $2$.

Suppose $X$ has $e^{\ii\zeta}(\cos\pi/4, \pm\ii\sin\pi/4)$-revival at time $\pi/2^{k}$.  
Then by Theorem~\ref{thm:gfr_tight}
\[k \leq r_{\ell'} \quad  \text{where}\ 
\ell'=
\begin{cases}
\ell & \text{if $r_{\ell}<n$,}\\
\ell-1 & \text{if $r_{\ell}=n$.}
\end{cases}
\]
By Theorem~\ref{thm:frchar}, there exist integers $m$, $t$ and $\mu$ such that
\[
\frac{t\mu\pi}{m} \equiv \pm \frac{\pi}{4} \pmod{2\pi}, \quad \frac{2t\pi}{mg}=\frac{\pi}{2^k} \quad \text{and}\quad \gcd(\mu,m)=1.
\]
Hence $m/t=4$, $g=2^{k-1}$ and, by Theorem~\ref{thm:frchar_gh}, $hg \equiv 0 \pmod{2^{k+1}}$.

Since
\[
\theta_0-\theta_s = 
\begin{cases}
(\theta_0-\theta_1) + (\theta_1-\theta_3) + \cdots + (\theta_{s-2}-\theta_s) & \text{if $s$ is odd,}\\
(\theta_0-\theta_2)  + \cdots + (\theta_{s-2}-\theta_s)  & \text{if $s$ is even,}
\end{cases}
\]
we have 
\[
g=\gcd( \theta_0-\theta_1, gh) = 2^{k-1}.
\]
Consequently,
\[\theta_0-\theta_1 = 2\sum_{i=1}^{\ell} \binom{n-1}{r_i-1} \equiv \pm 2^{k-1} \pmod{2^{k+1}}
\]
and Condition~(\ref{eqn:uno}) holds.

Condition~(\ref{eqn:dos}) follows directly from 
the expression of $hg$ given in Lemma~\ref{lem:gcd} and  $k \leq r_{\ell'}$.   Condition~(\ref{eqn:tres}) comes from the expression of $\zeta$ given in Theorem~\ref{thm:frchar}.

Conversely, if Conditions~(\ref{eqn:uno}) to (\ref{eqn:tres}) hold then $g=2^{k-1}$, $hg \equiv 0 \pmod{2^{k+1}}$ and
\[
\frac{\theta_0-\theta_1}{g} \equiv \mu \pmod{4} \quad \text{for some odd integer $\mu$.}
\]
Hence $m=4$, $t=1$ and $\mu$ satisfy Theorem~\ref{thm:frchar}.
\end{proof}

\subsection{Distance Graphs}
In this section, we focus on balanced fractional revival on a single distance graph in $\HH(n,2)$. In particular, we characterize all distance graphs $X_r$ that admit balanced fractional revival at time $\pi/4$ and $\pi/8$, respectively, in terms of $n$ and $r$. 

We first cite a useful result from number theory.

\begin{theorem} \label{thm:kummer}
(Kummer; see Dickson \cite{dickson}, page 270) \\
Let $p$ be a prime. The largest integer $k$ so that $p^{k}$ divides $\binom{n}{m}$ is the
number of carries in the addition of $n-m$ and $m$ in base $p$ representation.
\qed
\end{theorem}

Recall that, for positive integer $m$, $\alpha_2(m)$ denotes the largest $k$ such that $2^k$ divides $m$. Let $(m)_2$ denote the binary representation of $m$. The next result gives a necessary condition for balanced fractional revival to occurs at $X_r$.

\begin{lemma}
Let $X_r$ be a connected graph in $\mathcal{H}(n,2)$. If $X_r$ has balanced fractional revival, then $n$ is odd, and $\alpha_2(n-1) = \alpha_2(r-1)$.
\end{lemma}

\begin{proof}
If $X_r$ has balanced fractional revival, then by Proposition~\ref{prop:fr_hamming}
\[
\alpha_2(\binom{n-1}{r-1}) = k-2 \quad \text{and}\quad \alpha_2(\binom{n-2}{n-1}) \geq k-1, \quad \text{for some $k \leq r$.}
\]
First suppose $n$ is even. Since $X_r$ is connected, $r$ is odd. Thus the last digit of $(n-1-(r-1))_2$ is $1$, while the last digit of $(r-1)_2$ is $0$. Therefore, the number of carries when adding $(n-r)_2$ with $(r-1)_2$ is equal to the number of carries when adding $(n-r-1)_2$ with $(r-1)_2$. Thus,
\[\alpha_2(\binom{n-1}{r-1}) = \alpha_2(\binom{n-2}{r-1}),\]
and balanced fractional revival does not occur in $X_r$.

To see $\alpha_2(n-1) = \alpha_2(r-1)$, it suffices to show that $\alpha_2(n-1-(r-1))>\alpha_2(r-1)$. 
Suppose otherwise. Then there is an $m<\alpha_2(r-1)$ such that the $m$-th digit of $(n-2-(r-1))_2$ is $1$. It follows that the number of carries when adding $(n-r-1)_2$ with $(r-1)_2$ is less than or equal to the number of carries when adding $(n-r)_2$ with $(r-1)_2$, hence $\alpha_2(\binom{n-2}{r-1}) \leq \alpha_2(\binom{n-1}{r-1})$ and $X_r$ does not have balanced fractional revival.
\end{proof}

For balanced fractional revival times of $\pi/4$ and $\pi/8$, we
obtain tight characterizations.

\begin{proposition}
$X_r$ is a connected graph in $\HH(n,2)$ with balanced fractional revival at time $\pi/4$ if and only if the following hold.
\begin{enumerate}[(i)]
    \item $n$ is odd, and $n-1$ is not a power of 2.
    \item $(r-1)_2$ is obtained from $(n-1)_2$ by replacing some $1$'s with $0$'s, except at the $\alpha_2(n-1)$-th position.
 \end{enumerate}
\end{proposition}

\begin{proof}
By Proposition \ref{prop:fr_hamming}, $X_r$ has balanced fractional revival at time $\pi/4$ if and only if the following hold:
\begin{enumerate}[(a)]
    \item $\binom{n-1}{r-1}$ is odd, and
    \item $\binom{n-2}{r-1}$ is even.
\end{enumerate}
Suppose $X_r$ is a connected graph in $\HH(n,2)$ with fractional revival at time $\pi/4$. Then $r$ is odd and $r<n$. 

To see (i), assume for a contradiction that $n-1$ is a power of $2$. Then $(n-1)_2$ has exactly one digit of $1$, that is, the leading digit. By (a), for every $\ell$ such that the $\ell$-th digit of $(n-1)_2$ is $0$, the $\ell$-th digit of $(r-1)_2$ must also be $0$. Thus $r=1$, which contradicts the fact that $X_1$ does not have balanced fractional revival.

Next we prove (ii). Clearly, $\alpha_2(n-1)$ is the rightmost position at which $(n-1)_2$ is $1$. We have
\begin{align*}
(n-1)_2 &=1 \cdots 1 0 0 \cdots 0\\
(n-2)_2 &=1 \cdots 0 1 1 \cdots 1
\end{align*}
Thus, for both (a) and (b) to hold, $(r-1)_2$ must be $1$ at the $\alpha_2(n-1)$-th digit, and $0$ wherever $(n-1)_2$ is $0$.

The converse statement follows from reversing the argument.
\end{proof}

Recall that $X_1$ has no fractional revival and $X_2$ is not connected. In contrast, there are infinite many schemes where $X_3$ has fractional revival, as a corollary to the above result.
\begin{corollary} \label{cor:3mod4}
For $n \equiv 3\pmod{4}$,
$X_{3} \in \HH(n,2)$ has balanced fractional revival at $\pi/4$.
\qed
\end{corollary}

\begin{lemma}
$X_r$ is a connected graph in $\HH(n,2)$ with balanced fractional revival at time $\pi/8$ if and only if the following hold.
\begin{enumerate}[(i)]
    \item $n$ is odd.
    \item $\alpha_2(n-1) = \alpha_2(r-1)$.
    \item $\binom{n-1}{r-1} \equiv 2 \pmod{4}$.
\end{enumerate}
\end{lemma}

\begin{proof}
By Proposition \ref{prop:fr_hamming}, $X_r$ has balanced fractional revival at time $\pi/8$ if and only if 
\begin{enumerate}[(a)]
\item $\alpha_2(\binom{n-1}{r-1})=1$,
\item $\alpha_2(\binom{n-2}{r-1})\ge 2$,
\item $\alpha_2(\binom{n-3}{r-2})\ge 1$.
\end{enumerate}
Since $\binom{n-3}{r-2}+\binom{n-3}{r-1}=\binom{n-2}{r-1}$, (b) and (c) hold if and only if (b) and $\alpha_2(\binom{n-3}{r-1})\ge 1$. On the other hand, since $n$ and $r$ are odd, $n-r-1$ is odd, so
\[\alpha_2(\binom{n-2}{r-1})=\alpha_2(\binom{n-3}{r-1}).\]
Therefore, (a), (b) and (c) hold if and only if (a) and (b) hold. 

Finally, since
\[\binom{n-2}{r-1}+\binom{n-2}{r-2}=\binom{n-1}{r-1},\]
(a) and (b) hold if and only if (a) holds and 
\[\alpha_2(\binom{n-2}{r-2})=1.\]
Given that $\alpha_2(n-1)=\alpha_2(r-1)$, (a) holds if and only if 
\[\alpha_2(\binom{n-2}{r-2})=1.\]
Therefore (i), (ii) and (iii) are necessary and sufficient conditions for $X_r$ to be connected and have balanced fractional revival at time $\pi/8$.
\end{proof}

\begin{proposition}
$X_r$ is a connected graph in $\HH(n,2)$ with balanced fractional revival at time $\pi/8$ if and only if the following hold.
\begin{enumerate}[(i)]
    \item $n$ is odd, and $n-1\ne 2^a(2^b-1)$ for any non-negative integers $a$ and $b$.
    \item $(r-1)_2$ is obtained from $(n-1)_2$ by \begin{enumerate}
        \item keeping the $j$-th digit, for $j=0,1,\cdots,\alpha_2(n-1)$;
        \item replacing exactly one substring ``$10$" with ``$01$";
        \item replacing some ``$1$"s with ``$0$"s.
        \end{enumerate}
\end{enumerate}
\end{proposition}

\begin{proof}
First recall that for positive integers $a>b$, 
\[\binom{a}{b}\equiv 2 \pmod{4}\]
if and only if there is exactly one carry, say to the $\ell$-th position, when adding $(a-b)_2$ and $(b)_2$. Note that this happens if and only if 
\begin{itemize}
\item the $(\ell-1)$-th digit of $(a)_2$ is $0$, the $(\ell-1)$-th digit of $(b)_2$ is $1$;
\item the $\ell$-th digit of $(a)_2$ is 1, the $\ell$-th digit of $(b)_2$ is $0$;
\item $(a)_2$ is larger than $(b)_2$ at all other digits.
\end{itemize}
If in addition, $\alpha_2(a)=\alpha_2(b)$, then $\ell-1>\alpha_2(a)$, so $(a)_2$ cannot be a string of $1$'s followed by a string of $0$'s. Applying the above argument to $a=n-1$ and $b=r-1$ yields (i) and (ii).
\end{proof}

\begin{corollary} \label{cor:11mod16}
For $n \equiv 11\pmod{16}$,
$X_{7} \in \HH(n,2)$ has balanced fractional revival at $\pi/8$.
\qed
\end{corollary}

In what follows, we show that for every integer $k \ge 2$, there are families of distance graphs which exhibit fractional revival at time $\pi/2^k$ in some binary Hamming scheme.
This should not be interpreted to mean that fractional revival could be
implemented increasingly faster by growing $k$; physically the transport
time is also proportional to the inverse of the scaling factor of the
adjacency matrix (the Hamiltonian) that must decrease with $n$ to keep
the energy bounded (see the remark in \cite{cvz17}).
However, in order to focus on the combinatorial nature
of our constructions, we opt to work with unnormalized matrices.

\begin{proposition}
For $k \ge 4$, let $n = (2^{k-1}+1)2^{k+2}+3$ and $r = 2^{k+3}+3$.
Then $X_{r} \in \HH(n,2)$ has balanced fractional revival at time $\pi/2^{k}$.
\end{proposition}
\begin{proof}
We first show that Condition~(\ref{eqn:uno}) of Proposition~\ref{prop:fr_hamming} holds.
By Kummer's theorem, it suffices to show there are $k-2$ carries generated in the binary addition of $n-r$ and $r-1$.
We write the binary representations of the numbers:
\begin{eqnarray}
(n-1)_{2} & = & 10^{k-3} 0 \cdot 10^{k}10 \\
(r-1)_{2} & = & 00^{k-3} 1 \cdot 00^{k}10 \\
(n-r)_{2} & = & 01^{k-3} 1 \cdot 10^{k}00
\end{eqnarray}
Note there are exactly $k-2$ carries in the binary addition of $n-r$ and $r-1$, and
\[
\binom{n-1}{r-1}\equiv \pm 2^{k-2} \pmod{2^k}.
\]

To show that Condition~(\ref{eqn:dos}) of Proposition~\ref{prop:fr_hamming} holds, it suffices to show that there are at least $k-j$ carries generated in the binary addition of $n-r-1$ and $r-j$, for $j=1,\ldots,k-1$.
The binary representations of $n-r-1$ and of $r-1$, $r-2$, and $r-3$ are
\begin{eqnarray}
(n-r-1)_{2} & = & 01^{k-3} 1 \cdot 01^{k}11 \\
(r-1)_{2} & = & 00^{k-3} 1 \cdot 00^{k}10 \\
(r-2)_{2} & = & 00^{k-3} 1 \cdot 00^{k}01 \\
(r-3)_{2} & = & 00^{k-3} 1 \cdot 00^{k}00 
\end{eqnarray}
Note that at least $k-j$ carries are generated in the binary addition of $n-r-1$ and $r-j$, where
$j=1,2,3$.
Now, we consider the binary representations of $n-r-1$ and of $r-j$, for $j =4,\ldots,k-1$:
\begin{eqnarray}
(n-r-1)_{2} & = & 01^{k-3} 1 \cdot 01^{k}11 \\
(r-4)_{2} & = & 00^{k-3} 0 \cdot 11^{k}11 
\end{eqnarray}
Note there are at least $k+2$ carries in the binary addition of $n-r-1$ and $r-4$.
Moreover, the number of ones in the binary representation of $r-j$ decreases by at most
one as $j$ increases from $4$ to $k-1$.
Therefore, there are at least $k+2-(j-4) \ge k-j$ carries generated in the binary addition of
$n-r-1$ and $r-j$, for $j=1,\ldots,k-1$.

So, by Proposition \ref{prop:fr_hamming}, $X_{r}$ has fractional revival at time $\pi/2^{k}$.
\end{proof}


\subsection{Consecutive Unions}

Now we proceed to construct consecutive unions of distance graphs that admit balanced fractional revival at time $\pi/4$.

\begin{proposition}
For every $r\ge 2$, let $\ell$ be the positive integer such that $2^{\ell-1} <r \le 2^{\ell}$. Then for $n\equiv r+1 \pmod{2^{\ell}}$, the consecutive union $X_1 \cup X_2 \cup \cdots \cup X_r$ has balanced fractional revival at time $\pi/4$.
\end{proposition}

\begin{proof}
By Proposition \ref{prop:fr_hamming}, it suffices to prove the following:
\begin{enumerate}[(i)]
    \item $\sum_{j=1}^{r} \binom{n-1}{j-1}$ is odd, and
    \item $\sum_{j=1}^{r} \binom{n-2}{j-1}$ is even.
\end{enumerate}

Let $a$ and $b$ be two positive integers with $a\ge b$. From Theorem \ref{thm:kummer}, we see that $\binom{a}{b}$ is even if and only if there is at least one $i$ such that the $i$-th digit of $(a)_2$ is $0$ and the $i$-th digit of $(b)_2$is $1$. Thus, if $c$ is an integer with $2^c \ge a$, then for any positive integer $d$ we have that $\binom{a+2^c d}{b}$ is even if and only if $\binom{a}{b}$ is even. 

Now, since $n-1 \equiv r \pmod{2^{\ell}}$, it follows that 
\[\sum_{j=1}^r\binom{n-1}{j-1} \equiv \sum_{j=1}^r\binom{r}{j-1}\equiv 1 \pmod{2},\]
and
\[\sum_{j=1}^r \binom{n-2}{j-1} \equiv \sum_{j=1}^r \binom{r-1}{j-1} \equiv 0\pmod{2}.\]
Hence conditions (i) and (ii) hold.
\end{proof}


\section{Weighted Graphs}

In this section, we turn our attention to balanced fractional revival in graphs (possibly weighted) which lie in the span of the Hamming graphs. We first give a characterization of balanced fractional revival that occurs at time $\pi/\Omega$.
\begin{proposition} \label{prop:gfr_tight}
For an integer $n \ge 2$,
let $X$ be a graph whose adjacency matrix $A$ is in the Bose-Mesner algebra of $\HH(n,2)$.
Suppose the eigenvalues of $A$ are $\theta_0, \theta_1,\ldots, \theta_n$. 
Then $X$ has $e^{\ii\zeta}(\cos\pi/4, \pm\ii\sin\pi/4)$-revival at time $\pi/\Omega$
if and only if 
for some integers $h,h_{0},\ldots,h_{n-2}$,
we have
\begin{eqnarray}
\theta_{0} - \theta_{1} & = & \left(h + \tfrac{1}{2}\right)\Omega \\
\theta_{s} - \theta_{s+2} & = & 2h_{s}\Omega, \hspace{0.5in} \mbox{ for $s=0,\ldots,n-2$} 
\end{eqnarray}
and $\zeta + \theta_{0}\pi/\Omega \equiv \mp \pi/4 \pmod{2\pi}$.
\end{proposition}

\begin{proof}
Let $A = \sum_{s=0}^{n} \theta_{s}E_{s}$ where $E_{s}$ are the minimal idempotents of the scheme.
Note $A_{0} = \sum_{s} E_{s}$ and $A_{n} = \sum_{s} (-1)^{s}E_{s}$.
Suppose $e^{-\ii A\pi/\Omega} = e^{\ii\zeta}(\cos(\pi/4) A_{0} \pm \ii \sin(\pi/4) A_{n})$.
Then, for $s=0,\ldots,n$, we have
\begin{equation}
e^{-\ii\theta_{s}\pi/\Omega} 
	= e^{\ii\zeta} (\cos(\pi/4) \pm \ii\sin(\pi/4)(-1)^{s}).
\end{equation}
For $s=0$, we get $\zeta \equiv -\theta_{0}\pi/\Omega \mp \pi/4 \pmod{2\pi}$.
We have 
\[e^{-\ii(\theta_{0}-\theta_{1})\pi/\Omega}=\mp \ii\]
and
\[
e^{-\ii(\theta_{s}-\theta_{s+2})\pi/\Omega} = 1,
\quad \text{for $s=0,\ldots,n-2$.}
\]
So, there exist integers $h, h_{0},\ldots,h_{n-2}$ satisfying
\begin{eqnarray}
\theta_{0}-\theta_{1} & = & \left(h + \tfrac{1}{2}\right)\Omega \\
\theta_{s}-\theta_{s+2} & = & 2h_{s}\Omega, \qquad \text{for $s=0,\ldots,n-2$.}
\end{eqnarray}
This yields the claim.
\end{proof}

\subsection{Revisiting $\spn\{A_1,A_2\}$}

Fractional revival was first studied on weighted paths \cite{gvz16}. In \cite{cvz17} and \cite{bcltv18}, the authors constructed analytically examples with fractional revival in $\spn\{A_1, A_2\}$ of the Hamming scheme. We rediscover some of their results in this section.

The spectra of the Hamming graphs $X_{1}$ and $X_{2}$ are given by the Krawtchouk polynomials
$p_{1}(s) = n-2s$ and
$p_{2}(s) = \binom{n}{2} - 2s(n-1) + 4\binom{s}{2}$, respectively.
Using Proposition \ref{prop:gfr_tight}, we may classify the values $\omega_{1},\omega_{2}$ for which the graph
$\omega_{2} A_{2} + \omega_{1} A_{1}$ has fractional revival.
Recall that $A_{1}$ alone has no fractional revival \cite{cvz17}.
We shall consider two cases based on whether $\omega_{1}$ is zero or not.


First, we consider the case when $\omega_{1} \neq 0$. Here, we may consider instead
$\tilde{A} = \omega A_{2} + A_{1}$, where $\omega = \omega_{2}/\omega_{1}$. 
Let $\tilde{A} = \sum_{s=0}^{n} \theta_{s}E_{s}$ with
$\theta_{s} = \omega p_{2}(s) + p_{1}(s)$.
Note that
\begin{eqnarray} \label{eqn:spacing}
\theta_{0} - \theta_{1} & = & 2(\omega(n-1) + 1) \\
\theta_{s} - \theta_{s+2} & = & 4\big(\omega(n-2s-2) + 1\big).
\end{eqnarray}
By Proposition \ref{prop:gfr_tight}, for fractional revival to occur at time $\pi/\Omega$, 
it suffices to require
\begin{equation} \label{eqn:inductive}
2(\omega(n - 2s - 2) + 1) = \ZZ\Omega
\end{equation}
\begin{equation} \label{eqn:base}
2(\omega(n-1) + 1) = (\ZZ + \tfrac{1}{2})\Omega.
\end{equation}
Taking the difference of the last two equations, for each $s$, we have
\begin{equation} \label{eqn:constraint1}
2(2s+1)\omega/\Omega = \ZZ + \tfrac{1}{2}.
\end{equation}
This shows $\omega/\Omega$ is rational.
Moreover, $4\omega/\Omega$ must be an odd integer.

\begin{proposition}
In $\HH(n,2)$, $\omega A_2+A_{1}$ has balanced fractional revival at time $\pi/\Omega$ 
if and only if $4\omega/\Omega$ is an odd integer and $4/\Omega$ is an integer that has the same parity as $n$.
\begin{proof}
Let $4\omega/\Omega=2m+1$ for some integer $m$.
If $4/\Omega$ and $n$ have the same parity, then $4(\omega(n-1)+1)/\Omega = 2h+1$
for some integer $h$.
Then the integers $h$ and $h_s= h-2ms-s-m$, $s=0,\ldots,n-2$, satisfy the equations in
Proposition~\ref{prop:gfr_tight}.

The converse follows from Equations~(\ref{eqn:base}) and (\ref{eqn:constraint1}).
\end{proof}
\end{proposition}

\begin{corollary}
Suppose $\omega A_2+A_1$, for some $\omega\neq 0$, has balanced fractional revival at time $\pi/\Omega$. Then $\omega\in \QQ$ and $\pi/\Omega = q\pi/4$, 
for some $q$ with the same parity as $n$.
\qed
\end{corollary}

The following corollaries provide natural examples of {\em signed} (multi-)graphs 
in the Hamming scheme which exhibit fractional revival. 

\begin{corollary}
For any integer $m$, $A_{2} \pm 2A_{1}$ in the Bose-Mesner algebra of $\HH(2m,2)$
has balanced fractional revival at time $\pi/4$.  
(Also, $\frac{1}{2} A_2 \pm A_1$ has balanced fractional revival at time $\pi/2$.)
\qed
\end{corollary}

\begin{corollary}
For any integer $m$, $A_{2} \pm A_{1}$ in the Bose-Mesner algebra of $\HH(2m+1,2)$
has balanced fractional revival at time $\pi/4$.
\qed
\end{corollary}


Finally, we consider $\omega_{2}A_{2} + \omega_{1}A_{1}$ when $\omega_{1} = 0$.
The graph $A_2 \in \HH(n,2)$ has two connected components, one with vertices 
of even weights while the other consists of vertices of odd weights. We call the component containing 
vertices with even weights the {\em halved $n$-cube} \cite{bcn89}.

\begin{proposition}
In $\HH(n,2)$, $X_2$ has balanced fractional revival if and only if $n$ is even.
The time of balanced fractional revival is $\pi/4$.
\begin{proof}
From Proposition~\ref{prop:fr_hamming}, $X_2$ has balanced fractional revival at time $\pi/2^k$  if and only if
\[n-1 \equiv \pm 2^{k-2} \pmod{2^k} 
\quad\text{and} \quad n-2 \equiv 0 \pmod{2^{k-1}}.
\]
These equations hold exactly when $n$ is even and $k=2$.
\end{proof}
\end{proposition}

When $n$ is odd,  each vertex $a$ and its antipodal pair $\bar{a}$,
(that is, the neighbour of $a$ in $X_n$)
belong to different components of $A_2$; thus, fractional revival does not occur.
When $n$ is even,  fractional revival occurs on each connected component of $A_{2}$.
(See Bernard \etal \cite{bcltv18} for a similar treatment.)


\subsection{Larger Spans}

We consider balanced fractional revival on the larger spans involving the first four Hamming graphs,
namely, $\spn\{A_1,A_2,A_3\}$ and $\spn\{A_1,A_2,A_3,A_4\}$.

\begin{proposition}
For $n \equiv 3 \pmod{4}$, the sum $A_1 + \frac{1}{2} A_2 + \frac{1}{4} A_3$ in $\HH(n,2)$ has fractional revival at time $\pi$.
\end{proposition}

\begin{proof}
By Corollary \ref{cor:3mod4}, it suffices to show that $A_1$ and $A_2$ are periodic at time $\pi$ and $\pi/2$, respectively. Since 
\[p_1(s) = n-2s\]
and
\[p_2(s) = \binom{n}{2} - 2s(n-1) + 4\binom{s}{2},\]
we see that $2$ divides 
\[\gcd\big\{p_1(s) - p_1(0)\big\}_{s=0}^n,\]
and $4$ divides
\[\gcd\big\{p_2(s) - p_2(0)\big\}_{s=0}^n.\]
\end{proof}

\begin{proposition}
For $n\equiv 3\pmod{8}$, the sum $A_1 + \frac{1}{2} A_2 + \frac{1}{4} A_3 + \frac{1}{8} A_4$ in $\HH(n,2)$ has fractional revival at time $\pi$.
\end{proposition}

\begin{proof}
From the previous result, it suffices to show that $X_4$ is periodic at time $\pi/8$. Since 
\[p_4(s) = \sum_{j=0}^4 (-2)^j\binom{n-j}{4-j}\binom{s}{j},\]
we have
\[p_4(s) - p_4(0) = -2\binom{n-1}{3}s + 4\binom{n-2}{2}\binom{s}{2} - 8(n-3)\binom{s}{3} +16\binom{s}{4}.\]
If $n=8m+3$, then
\[
2\binom{n-1}{3} = \frac{(8m+2)(8m+1)(8m)}{3},\quad
4\binom{n-2}{2}=2(8m+1)(8m), \quad \text{and}\quad
8(n-3)
\]
are all divisible by $16$.
\end{proof}


\section{Conclusion}

The main achievement of this work is the characterization of graphs in association schemes admitting quantum fractional revival in terms of their spectra, and the discovery of several infinite families of graphs exhibiting quantum fractional revival. From the mathematical point of view, we are bridging the fields of algebraic graph theory, specifically what concerns association schemes and orthogonal polynomials, to elementary number theory and the study of certain periodic functions that arise naturally in quantum information. For quantum information theory, our findings may turn to be quite useful for entanglement generation procedures or other tasks that require a network of many interacting qubits to be put in a state with several pairs of maximally entangled qubits.

Following the work in this paper, we raise the following questions:
\begin{enumerate}[(1)]
    \item Study the more general notion of fractional revival among several vertices in association schemes. We are preparing an upcoming publication related to this topic. 
    \item Extend the results for the Hamming Scheme to other cubelike graphs.
\end{enumerate}


\section*{Acknowledgments}

The research of L.V. is supported by a discovery grant from
the National Science and Engineering Research Council (NSERC) of Canada.



\end{document}